\documentclass[10pt]{amsart}

\usepackage{amsmath}
\usepackage{amsfonts}
\usepackage{amssymb}
\usepackage{epsfig}
\usepackage{graphicx,psfrag}

\numberwithin{equation}{section}

\usepackage[T1]{fontenc}
\usepackage{palatino}

\newtheorem{lem}{Lemma}[section]
\newtheorem{thm}{Theorem}[section]

\newtheorem{obs}{Remark}[section]

\def\k0{\kappa_0}

\def\lgl{\langle}
\def\rgl{\rangle}

\def\bfe{{\mathsf{e}}}

\def\bfu{{\bf{u}}}

\def\bfx{{\bf{x}}}

\def\bfo{{\bf{0}}}

\def\mR3{{\mathbb{R}^3}}

\begin{document}
\title[Dissipation Anomaly]
{Dissipation anomaly and energy cascade in 3D incompressible flows
\\ \bigskip Dissipation anormale et cascade \'energ\'etique pour des fluides
incompressibles tridimensionnels }
\author{R. Dascaliuc}
\address{Department of Mathematics\\
Oregon State University\\ Corvallis, OR 97332}
\author{Z. Gruji\'c}
\address{Department of Mathematics\\
University of Virginia\\ Charlottesville, VA 22904}
\date{\today}

\maketitle

\noindent ABSTRACT \ The purpose of this note is to present a
mathematical evidence of \emph{dissipation anomaly} in 3D turbulent
flows within a general setting for the study of energy cascade in
\emph{physical scales} of 3D incompressible flows recently
introduced by the authors.

\bigskip

\noindent R\'ESUM\'E \ Le but de cette note est de pr\'esenter une
mise en \'evidence math\'ematique de la {\em dissipation anormale}
pour des flots turbulents tridimensionnels, dans un cadre
g\'en\'eral r\'ecemment introduit par les auteurs pour l'\'etude de
la cascade \'energ\'etique aux {\em \'echelles physiques} dans les
fluides incompressibles en trois dimensions.

\section{Introduction}

\noindent \emph{Dissipation anomaly}, i.e., non-vanishing of
averaged energy dissipation rate in the infinite Reynolds number
limit (sometimes referred to as `zeroth law of turbulence'), is a
key player in both empirical and phenomenological turbulence. On one
hand, documented empirical confirmation dates back at least to
Dryden's experiments on decaying turbulence in wind tunnels reported
in 1943, and on the other hand, dissipation anomaly was a key
postulate in Kolmogorov K41 phenomenology
\cite{Ko41-1}
(see also a discussion in \cite{Fr95,E08,Shvy10}).
K41 predicts existence
of \emph{energy cascade} -- a local (in scale) nonlinear transfer of
averaged energy across a range of scales, the so-called inertial
range. In addition, the transfer rate/averaged energy flux is
expected to be nearly-constant, and equal to averaged energy
dissipation rate. Consequently, dissipation anomaly can be recast in
terms of averaged energy flux as non-vanishing of averaged energy
flux in the infinite Reynolds number limit (across a suitable range
of scales).

Onsager in 1945-49 \cite{On49,ES06}
conjectured that irregular/singular solutions to
the 3D Euler equations (the inviscid model) are capable of
effectively dissipating all the energy in the flow. More precisely,
Onsager conjectured that a minimal spatial regularity of a (weak)
solution to the 3D Euler needed to conserve the energy is of the
order of $\bigl(\frac{1}{3}\bigr)^+$
(the main mathematical works confirming
Onsager's $\frac{1}{3}$ criticality conjecture are
\cite{E94,CET94,DR00,CCFS08}),
and that in the case the energy is not
conserved, the dissipation due to the singularities --
\emph{anomalous dissipation} -- triggers the energy cascade which
then continues indefinitely (unlike in the viscous case, the
inertial range is expected to extend all the way to the
zero-scale). This offers a way of studying dissipation anomaly in
terms of vanishing viscosity limits of solutions to the 3D
Navier-Stokes equations (NSE) converging to a singular solution to
the 3D Euler exhibiting anomalous dissipation.

In a recent work, the authors introduced a setting for the study of
energy cascade in 3D viscous incompressible flows in \emph{physical
scales}. (The preexisting work on existence of 3D energy cascade
\cite{FMRT01} took place in the Fourier space.) This led to a proof
-- under a suitable condition plausible in the regions of intense
fluid activity -- of both existence and scale-locality of the energy
cascade in physical scales in decaying turbulence directly from the
3D NSE \cite{DaGr11-1}. The general mathematical setting is the one
of weak solutions satisfying the \emph{local energy inequality}, and
the key ingredient in the proof is suitable \emph{ensemble
averaging} of the local energy inequality. This method was then
adopted to the 3D inviscid flows to show -- under an assumption that
the anomalous dissipation in the region of interest is strong enough
(with respect to the energy) -- both existence and scale-locality of
the energy cascade in the 3D Euler flows extending \emph{ad
infinitum} confirming Onsager's predictions \cite{DaGr11-2}. (One
should note that the inviscid result is obtained in the setting of
(at this point in time) \emph{hypothetical} weak solutions to the 3D
Euler satisfying the local energy inequality; a natural regularity
here is $L^3$-locally in both space and time.)

The aim of this note is to briefly review the aforementioned results
-- providing a unified approach to both viscous and inviscid cases
-- and observe that the inviscid result leads to a short proof of
dissipation anomaly in this setting. Consider an $L^3$-strong in the
space-time vanishing viscosity limit of weak solutions to the 3D NSE
satisfying the local energy inequality converging to a
(hypothetical) $L^3$ in the space-time weak solution to the 3D
Euler. Duchon and Robert \cite{DR00} noticed that the Euler solution
then necessarily satisfies the local energy inequality. In fact,
their calculation implies, assuming that the NSE solutions do not
exhibit anomalous dissipation \emph{per se}, that the local
spatiotemporal viscous energy dissipation rates converge to the
local spatiotemporal anomalous dissipation corresponding to the
vanishing viscosity Euler solution, resulting in a local
spatiotemporal dissipation anomaly (provided the inviscid anomalous
dissipation is strictly positive). Here, we show that our approach
to the study of energy cascade in physical scales of 3D
incompressible flows leads to a \emph{scale-to-scale} manifestation
of spatiotemporal dissipation anomaly. More precisely, denoting the
macro-scale in the problem by $R_0$, it is shown -- provided the
condition for the inviscid cascade holds -- that for every $R^*$,
$0<R^*<R_0$, there exists $\nu^*$ positive, such that for any $\nu$,
$0<\nu<\nu^*$, the averaged energy flux associated with the viscous
solution $u^\nu$ is nearly-constant and comparable to the anomalous
dissipation \emph{throughout the inertial range} delineated by
$[R^*,R_0]$. Moreover, the sufficient condition for the never-ending
anomalous energy cascade in the inviscid case triggers the
sufficient conditions for the energy cascade in the vanishing
viscosity sequence with ever-expanding inertial ranges.

\section{3D viscous energy cascade revisited}

\noindent Let $\bfx_0$ be in $B(\bfo,R_0)$ ($R_0$ being a given
arbitrary length -- interpreted as a macro-scale, such that
$B(\bfo,2R_0)$ is contained in $\Omega$ where $\Omega$ is the global
spatial domain) and $0< R \le R_0$. For the exact localization
procedure at scale $R$, around the point $\bfx_0$ -- as well as the
precise interpretation of the local (kinetic) energy transfer in
this setting --  the reader is referred to \cite{DaGr11-1,
DaGr11-2}.

Let $\bfu$ be a weak solution to the 3D NSE on $\Omega \times
(0,2T)$ satisfying the \emph{local energy inequality},
\begin{equation}\label{loc_nse}
\nu \iint |\nabla \bfu|^2 \phi  \le \iint \frac{1}{2}|\bfu|^2 (\partial_t \phi + \nu
\triangle \phi) + \iint \biggl(\frac{1}{2}|\bfu|^2+p\biggr)\bfu \cdot \nabla \phi,
\end{equation}
for any nonnegative test function $\phi$ (e.g., a {\em suitable weak
solution}). Physically, the term ${\nu} \iint |\nabla \bfu|^2 \phi$
represents the local spatiotemporal energy dissipation rate due to
viscosity, while a (non-negative) defect in the local energy
inequality due to possible singularities/lack of smoothness can be
interpreted as the local spatiotemporal anomalous dissipation. For a
refined cut-off function $\phi=\phi_{\bfx_0,R,T}$ defined on
$B(\bfx_0,2R) \times (0,2T)$ denote by $\Phi_{\bfx_0,R}$ a local
\emph{inward} flux,

\begin{equation}\label{fluxdef}
\Phi_{\bfx_0,R}(t)= - \int \biggl((\bfu \cdot \nabla)\bfu + \nabla
p\biggr) \cdot \bfu \, \phi \, d\bfx = \int
\biggl(\frac{1}{2}|\bfu|^2+p\biggr) \bfu \cdot \nabla \phi \, d\bfx,
\end{equation}
and by $\hat{\Phi}_{\bfx_0,R}$ a local time-averaged flux per unit
mass,
$\displaystyle{\hat{\Phi}_{\bfx_0,R,T}=\frac{1}{TR^3}\int\Phi_{\bfx_0,R}(t)}$.
Similarly, denote by $\hat{\bfe}_{\bfx_0,R}$ a local time-averaged
energy and by $\hat{\varepsilon}_{\bfx_0,R}$  a local time-averaged
total (viscosity plus anomalous) energy dissipation rate, all per
unit mass,
\begin{equation}\label{loc_def}
\hat{\bfe}_{\bfx_0,R,T}=\frac{1}{TR^3}\iint
\frac{1}{2}|\bfu|^2\phi^{2\delta-1}, \ \
\hat{\varepsilon}_{\bfx_0,R,T}=\frac{1}{TR^3}\iint
\frac{1}{2}|\bfu|^2 (\partial_t \phi  + \nu\triangle
\phi)+\hat{\Phi}_{x_0,R,T}\,
\end{equation}
(for a suitable $\delta, \, \frac{1}{2} < \delta \le 1$).

Also denote by $\bfe$ and $\varepsilon$ the time-averaged energy and
total energy dissipation rate per unit mass associated to the
macro-scale domain on $(0,2T)$,
\begin{equation}\label{glob_def}
\bfe=\hat{\bfe}_{\bfo,R_0,T}\quad\mbox{and}\quad \varepsilon=\hat{\varepsilon}_{\bfo,R_0,T}\,.
\end{equation}

In order to connect the inviscid case with the vanishing viscosity
limit, we are adopting a slightly different approach from the one
presented in \cite{DaGr11-1} where the defect in the local energy
inequality (\ref{loc_nse}) was interpreted as the anomalous energy
flux due to singularities, while the time interval $T$ was bounded
below by $R_0^2/\nu$. Here, besides considering total (viscous plus
anomalous) energy dissipation rate, we will be keeping the time
scale $T$ independent of $\nu$. However, the calculations are
completely analogous to the ones presented in \cite{DaGr11-1}. In
particular, the following version of Theorem 4.1 in \cite{DaGr11-1}
holds. For the precise definition of the ensemble averaging process
$\displaystyle{\{\langle \cdot \rangle_R\}_{0 < R \le R_0}}$ -- with
respect to '$(K_1,K_2)$-covers at scale $R$', the reader is referred
to \cite{DaGr11-2}; here, we briefly present the main idea.

Let $K_1$ and $K_2$ be two positive integers. A cover
$\{B(x_i,R)\}_{i=1}^n$ of $B(0,R_0)$ is a \emph{$(K_1,K_2)$-cover at
scale $R$} if $\displaystyle{\biggl(\frac{R_0}{R}\biggr)^3 \le n \le
K_1
 \biggr(\frac{R_0}{R}\biggr)^3}$,
and any point $x$ in $B(0,R_0)$ is covered by at most $K_2$ balls
$B(x_i,2R)$. $K_1$ and $K_2$ represent global and local
\emph{multiplicities}, respectively. For any physical density of
interest $\theta$, consider time-averaged, per unit mass --
spatially localized to the cover elements $B(x_i, R)$ -- local
quantities $\hat{\theta}_{x_i,R}$,
$\displaystyle{\hat{\theta}_{x_i,R}=\frac{1}{TR^3}\iint \theta \,
\phi_i^\rho}$ (for a suitable $\rho, 0 < \rho \le 1$), and denote by
$\langle\Theta\rangle_R$ the ensemble average given by
$\displaystyle{\langle\Theta\rangle_R = \frac{1}{n} \sum_{i=1}^n
 \hat{\theta}_{x_i,R}}$.
A key observation is that $\langle\Theta\rangle_R$ being
\emph{stable} -- nearly-independent on a particular choice of the
cover (with the fixed multiplicities $K_1$ and $K_2$) -- indicates
there are no significant fluctuations of the sign of the density of
$\theta$ at scales comparable or greater than $R$. Consequently, for
an \emph{a priori} sign-varying density (e.g., the flux density),
the ensemble averaging process acts as a \emph{coarse detector} of
the \emph{sign-fluctuations at scale $R$}.

\begin{thm}\label{viscous_thm}
Let $\tau$ be a modified Taylor length scale defined by
$\displaystyle{\tau_T =
\bigg[\Big(\frac{R_0^2}{T}+\nu\Big)\frac{e}{\varepsilon}\bigg]^{1/2}}$.
There exist positive constants $c$ and $K$ -- depending only on the
cover parameters $K_1$ and $K_2$ -- such that for any $\gamma$ in
$(0,1)$, the condition $\tau_T^2<\frac{\gamma}{c}R_0^2$ implies
\begin{equation}\label{visc_casc}
\frac{1}{K}(1-\gamma)\,\varepsilon\le\langle\Phi\rangle_R\le K(1+\gamma)\,\varepsilon
\end{equation}
for all $R$ inside the inertial range determined by $\displaystyle{
\left[\left(\frac{\tau_T^2}{\frac{\gamma}{c}(R_0^2+\nu
T)-\tau_T^2}\,\nu T\right)^{1/2},\, R_0\right]\,.}$
\end{thm}

\section{3D inviscid energy cascade}

In this section, we recall a sufficient condition for energy cascade in
3D inviscid flows obtained in \cite{DaGr11-2}, Section 3.
Let $\bfu$ be a (hypothetical) weak $L^3$ in the space-time
solution to the 3D Euler satisfying the local energy inequality,
\begin{equation}\label{loc_euler}
\iint \frac{1}{2}|\bfu|^2 \partial_t \phi  +
\iint \biggl(\frac{1}{2}|\bfu|^2+p\biggr)\bfu \cdot \nabla \phi = \varepsilon(\bfu;\phi) \ge 0
\end{equation}
($(\varepsilon(\bfu;\phi)$ is the anomalous dissipation associated
to the support of $\phi$). Keeping the notation in line with the
previous section -- focusing on $B(\bfx_0,R)$ -- denote by
$\hat{\epsilon}_{\bfx_0,R,T}$ a local time-averaged anomalous
dissipation per unit mass associated to $B(\bfx_0,R)$, and define
$\langle\Phi\rangle_R$, $\bfe$, and $\varepsilon$ as in the viscous
case, setting $\nu=0$. In the same spirit, define the anomalous
Taylor scale $\tau_T$ by
$\displaystyle{\tau_T=\biggl(\frac{R_0^2}{T}
\frac{\bfe}{\varepsilon}\biggr)^{1/2}\,.}$ The main result in
\cite{DaGr11-2} states the following.
\begin{thm}\label{inviscid_thm}
There exist positive constants $c$ and $K$ -- depending only on the
cover parameters $K_1$ and $K_2$ -- such that for any $\gamma$ in
$(0,1)$, the condition $\tau_T^2<\frac{\gamma}{c}R_0^2$ implies
\begin{equation}\label{invisc_casc}
\frac{1}{K}(1-\gamma)\,\varepsilon\le\langle\Phi\rangle_R\le K(1+\gamma)\,\varepsilon
\end{equation}
for all $R$ inside the inertial range determined by $\left(0,\,
R_0\right]$.
\end{thm}

Note that in the inviscid case, once the energy cascade commences, it continues
indefinitely towards the zero-scale as predicted by Onsager.

\section{Dissipation anomaly and energy cascade}

Let $\{\bfu^\nu\}$ be a family of weak solutions to the 3D NSE
satisfying the local energy inequality converging strongly in
$L^3\bigl(\mathbb{R}^3 \times (0,2T)\bigr)^3$ to a (hypothetical)
weak $L^3$ solution to the 3D Euler $\bfu$. As mentioned in
Introduction, Duchon and Robert \cite{DR00} noticed that the
inviscid limit $\bfu$ in this setting also satisfies the local
energy inequality (\ref{loc_euler}). A superscript $\nu$  will be
used to denote the energy, the flux and the (total) energy
dissipation rate corresponding to a solution $\bfu^{\nu}$, while the
quantities related to $\bfu$ will be denoted as in the previous
section. A straightforward computation -- utilizing the identity
$p=-\mathcal{R}_j\mathcal{R}_k(u^ju^k)$ on $\mathbb{R}^3$
($\mathcal{R}_l$ denoting the $l$-th Riesz transform;
$\bfu=(u^1,u^2,u^3)$\,) and the bound $|\nabla\phi_{\bfx,R,t}| \le
C_0/{R}$ -- implies the following estimate; this is simply a
quantitative version of the corresponding calculation in \cite{DR00}
written in our setting.

\begin{lem}\label{conv_lem}
Let $\bfx\in\mathbb{R}^3$, $R>0$ and $t\in(0,T]$. Then,
\begin{equation}\label{meep}
\bigl| \Phi^\nu_{\bfx,R,t} - \Phi_{\bfx,R,t} \bigr| \leq c
\frac{1}{t R^4} \|\bfu^\nu - \bfu\|^3_{L^3\bigl(\mathbb{R}^3 \;
\times (0,2T)\bigr)^3} .
\end{equation}
\end{lem}

\begin{thm}\label{diss_anomaly}
Suppose that the inviscid cascade condition,
$\tau_T<\sqrt{\gamma/c}R_0$, holds for the inviscid limit $\bfu$.
Then, for every $R^*$, $0<R^*<R_0$, there exists $\nu^*$ positive,
such that for any $\nu$, $0<\nu\leq\nu^*$,
\begin{equation}
 \frac{1}{K_\gamma} \varepsilon \leq \langle\Phi^\nu\rangle_R \leq K_{\gamma} \varepsilon
\end{equation}
throughout the inertial range determined by $[R^*,R_0]$, where
$\displaystyle{K_\gamma=\frac{K}{1-\gamma}}.$
\end{thm}

\begin{proof} Note that the inviscid cascade condition allows us to apply Theorem
\ref{inviscid_thm} with $\gamma:=\gamma+\delta$,
$0<\delta<c\tau_T^2/R_0^2\,-\,\gamma$, obtaining
\[
\frac{1}{K}{(1-\gamma-\delta)} \, \varepsilon \leq \langle\Phi\rangle_R \leq
\frac{K}{1-\gamma-\delta} \,  \varepsilon
\]
for any $R$, $0<R\leq R_0$. On the other hand, (\ref{meep}) implies
that for a fixed $(K_1,K_2)$-cover
\[\lgl\Phi^{\nu}\rgl_R\to\lgl\Phi\rgl_R \, ,\quad \mbox{as}\quad \nu\to 0\;\]
\emph{uniformly} on $[R^*,R_0]$ for any $R^*$, $0<R^*<R_0$.
Combining the two estimates finishes the proof.
\end{proof}

\begin{obs}{\em In fact, it is plain to verify that for all $\nu$, $0<\nu<\nu^*$, the energy cascade
\[
\frac{1}{K}(1-\gamma)\,\varepsilon^{\nu}\le\langle\Phi^{\nu}\rangle_R\le K(1+\gamma)\,\varepsilon^{\nu}
\]
holds inside the ever-expanding inertial ranges
$\displaystyle{\left[\left(\frac{(\tau^\nu_T)^2}{\frac{\gamma}{c}(R_0^2+\nu
T)-(\tau^{\nu}_T)^2}\,\nu T\right)^{1/2}, \, R_0\right]\; .}$ This
provides an \emph{intrinsic version of Theorem 4.1} -- without an
explicit reference to the vanishing viscosity limit (of course,
$\varepsilon^\nu \to \varepsilon$). }\end{obs}

\medskip

\noindent ACKNOWLEDGEMENTS \ The authors thank the anonymous referee
for a number of suggestions that improved the quality of the
presentation.

\end{document}